\documentclass[11pt,leqno]{amsart}
\usepackage{amsmath,amsfonts,latexsym,graphicx,amssymb,url, color,cite}
\usepackage{hyperref}
\usepackage[margin=1in]{geometry} 
\newcommand{\average}{{\mathchoice {\kern1ex\vcenter{\hrule
height.4pt width 6pt depth0pt} \kern-9.7pt}
{\kern1ex\vcenter{\hrule height.4pt width 4.3pt depth0pt}
\kern-7pt} {} {} }}

\makeatletter
\def\l@subsection{\@tocline{2}{0pt}{2.5pc}{5pc}{}}
\makeatother

\begin{document}

\newcommand{\dist}{\text{dist}} 
\newcommand{\abs}[1]{\left\vert#1\right\vert}
\newcommand{\diam}{\text{diam}}
\newcommand{\trace}{\text{trace}}
\newcommand{\R}{{\mathbb R}} 
\newcommand{\C}{{\mathbb C}}
\newcommand{\Z}{{\mathbb Z}}
\newcommand{\N}{{\mathbb N}}
\newcommand{\gradg}{\nabla_{\G}}
\newcommand{\e}{\varepsilon} 
\newcommand{\p}{\partial}

\newcommand{\calW}{{\mathrm W}}

\newcommand{\Rn}{\mathbb R^n}
\newcommand{\Rm}{\mathbb R^m}
\renewcommand{\L}[1]{\mathcal L^{#1}}
\newcommand{\G}{\mathbb G}
\newcommand{\calL}{\mathcal L}
\newcommand{\U}{\mathcal U}
\newcommand{\M}{\mathcal M}
\newcommand{\eps}{\epsilon}
\newcommand{\BVG}{BV_{\G}(\Omega)}
\newcommand{\no}{\noindent}
\newcommand{\ro}{\varrho}
\newcommand{\rn}[1]{{\mathbb R}^{#1}}
\newcommand{\res}{\mathop{\hbox{\vrule height 7pt width .5pt depth 0pt
\vrule height .5pt width 6pt depth 0pt}}\nolimits}
\newcommand{\hhd}[1]{{\mathcal H}_d^{#1}} \newcommand{\hsd}[1]{{\mathcal
S}_d^{#1}} \renewcommand{\H}{\mathbb H}
\newcommand{\BVGL}{BV_{\G,{\rm loc}}}
\newcommand{\GH}{H\G}
\renewcommand{\diam}{\mbox{diam}\,}
\renewcommand{\div}{\mbox{div}\,}
\newcommand{\divg}{\mathrm{div}_{\G}\,}
\newcommand{\norm}[1]{\|{#1}\|_{\infty}} \newcommand{\modul}[1]{|{#1}|}
\newcommand{\per}[2]{|\partial {#1}|_{\G}({#2})}
\newcommand{\Per}[1]{|\partial {#1}|_{\G}} \newcommand{\scal}[3]{\langle
{#1} , {#2}\rangle_{#3}} \newcommand{\Scal}[2]{\langle {#1} ,
{#2}\rangle}
\newcommand{\fron}{\partial^{*}_{\G}}
\newcommand{\hs}[2]{S^+_{\H}({#1},{#2})} \newcommand{\test}{\mathbf
C^1_0(\G,\GH)} \newcommand{\Test}[1]{\mathbf C^1_0({#1},\GH)}
\newcommand{\card}{\mbox{card}}
\newcommand{\bom}{\bar{\gamma}}
\newcommand{\shpiu}{S^+_{\G}}
\newcommand{\shmeno}{S^-_{\G}}
\newcommand{\CG}{\mathbf C^1_{\G}}
\newcommand{\CH}{\mathbf C^1_{\H}}
\newcommand{\di}{\mathrm{div}_{X}\,}
\newcommand{\norma}{\vert\!\vert}
\newenvironment{myindentpar}[1]%
{\begin{list}{}%
         {\setlength{\leftmargin}{#1}}%
         \item[]%
}
{\end{list}}

\renewcommand{\baselinestretch}{1.2}

\theoremstyle{plain}
\newtheorem{theorem}{Theorem}[section]
\newtheorem{corollary}[theorem]{Corollary}
\newtheorem{lemma}[theorem]{Lemma}
\newtheorem{proposition}[theorem]{Proposition}
\newtheorem{definition}[theorem]{Definition}
\newtheorem{rem}[theorem]{Remark}
\providecommand{\bysame}{\makebox[3em]{\hrulefill}\thinspace}

\renewcommand{\theequation}{\thesection.\arabic{equation}}

\title[Global H\"older estimates for 
linearized Monge--Amp\`ere ]{ Global H\"older estimates for 2D
linearized Monge--Amp\`ere equations  with right-hand side in divergence form}
\date{}
\author{Nam Q. Le}
\address{Department of Mathematics, Indiana University, Bloomington, IN 47405, USA}
\email{nqle@indiana.edu}
\thanks{The research of the author was supported in part by NSF grant DMS-1764248.}

\maketitle
\begin{abstract}
 We
 establish global H\"older estimates for solutions to inhomogeneous
linearized Monge--Amp\`ere equations in two dimensions with the right hand side being the divergence of a bounded vector field. These equations arise in the semi-geostrophic equations in meteorology
and in the approximation of convex functionals subject to a convexity constraint using fourth order Abreu type equations.
Our estimates hold under natural assumptions on the domain, boundary data and Monge-Amp\`ere measure being bounded away from zero and infinity. 
They are an affine invariant and degenerate version of 
global H\"older estimates by Murthy-Stampacchia and Trudinger for second order elliptic equations in divergence form. 

\end{abstract}

\setcounter{equation}{0}
\section{Introduction and statement of the main result}\label{intro}
In this paper, we  
establish global H\"older estimates for solutions to inhomogeneous
linearized Monge--Amp\`ere equations in two dimensions with the right hand side being the divergence of a bounded vector field; see Theorem \ref{global-reg}. Theorem \ref{global-reg} is an affine invariant and degenerate version of 
global H\"older estimates by Murthy-Stampacchia \cite{MuSt} and Trudinger \cite{Tr} for second order elliptic equations in divergence form with coefficient matrices together with their inverses having highly integrable eigenvalues. 
Our global H\"older estimates hold under natural assumptions on the domain, boundary data and Monge-Amp\`ere measure being bounded away from zero and infinity.
They are the global counterpart of the interior H\"older estimates recently established in \cite{LCMP} that we will recall in Theorem \ref{Holder_int_thm}. A crucial tool for our global H\"older estimates is the global $W^{1,1+\e}$ estimates for the Green's function of the linearized Monge-Amp\`ere operator in two dimensions established in Theorem \ref{DG_lem}. An application of Theorem \ref{global-reg} to solvability of singular, fourth order Abreu type equations will be 
presented in Theorem \ref{SBV3}.

Let $\Omega\subset \R^n$ ($n\geq 2$) be a bounded convex domain and let $\phi\in C^2(\Omega)$ be a locally uniformly convex function on  $\Omega$. The linearized  Monge-Amp\`ere equations corresponding to $\phi$ are of the form
\begin{equation}\label{LMA-eq}
\mathcal{L}_{\phi} u:=- \sum_{i, j=1}^{n} \Phi^{ij} u_{ij}= f\quad \mbox{in}\quad \Omega,
\end{equation}
where $$\Phi=\big(\Phi^{ij}\big)_{1\leq i, j\leq n} := (\det D^2 \phi )~ (D^2\phi)^{-1}
$$ is the cofactor matrix of the Hessian matrix $\displaystyle D^2\phi= (\phi_{ij})_{1\leq i, j\leq n}$. 
The operator $\mathcal{L}_\phi$ appears in
several contexts including affine differential geometry \cite{TW00}, complex geometry \cite{D05}, and fluid mechanics \cite{ACDF12, ACDF14, CNP91, Loe1, Loe}. 
In these contexts, one usually encounters the linearized Monge-Amp\`ere equations with 
the Monge-Amp\`ere measure $\det D^2\phi$ satisfying the pinching condition
\begin{equation}\lambda\leq \det D^{2} \phi\leq \Lambda.
 \label{pinch1}
\end{equation}
In this paper, we focus our attention to (\ref{LMA-eq}) under (\ref{pinch1}). Notice that since $\Phi$ is positive semi-definite, $\mathcal{L}_{\phi}$ is a linear elliptic partial differential operator, 
possibly both degenerate and singular.
Caffarelli 
and Guti\'errez initiated the study of the linearized Monge-Amp\`ere equations in the fundamental paper \cite{CG97}. 
There they developed an interior Harnack inequality theory for nonnegative solutions of the homogeneous
equation $\mathcal{L}_{\phi} u=0$ in terms of the pinching of the Hessian determinant in (\ref{pinch1}).
This theory is an affine invariant version of the classical Harnack inequality for linear, uniformly elliptic equations with measurable coefficients. As a consequence, they obtained interior H\"older estimates for the
homogeneous linearized Monge-Amp\`ere equation  $\mathcal{L}_{\phi} u=0$.

For the inhomogeneous equation (\ref{LMA-eq}) with $L^q$ right hand side $f$ where $q\geq 1$, 
Nguyen and the author \cite{LN3} recently established an interior Harnack inequality, interior H\"older estimates, and global H\"older estimates for solutions under natural 
conditions when $q>n/2$, which is the optimal range of $q$. The interior and global H\"older estimates, respectively, in \cite{LN3} rely heavily on 
the corresponding interior and global high integrability
of Green's function of the 
linearized Monge--Amp\`ere  operator $\mathcal{L}_\phi$ established in \cite{L, Le15}.

Regarding H\"older estimates, less is know about (\ref{LMA-eq}) when the right hand side $f$ is the divergence of a bounded vector field. This type of inhomogeneous linearized Monge-Amp\`ere equations arises in the semi-geostrophic equations in meteorology \cite{ACDF12, ACDF14, CNP91, LCMP, Loe1, Loe}.  
They also appear in second boundary value problems of fourth order equations of Abreu type arising from approximation of convex functionals whose Lagrangians depend on the gradient variable, subject to a convexity constraint; see \cite{Le18}. 
These functionals arise in different scientific disciplines such as Newton's problem of minimal resistance in physics and monopolist's problem in economics \cite{BFK, RC}.

In the case of the semi-geostrophic equations in meteorology, when the Monge-Amp\`ere measures $\det D^2\phi$ are continuous, interior H\"older estimates for solutions to (\ref{LMA-eq}) were obtained by Loeper \cite{Loe1}. However, when the measures $\det D^2\phi$ are only bounded away from
zero and infinity,  up to now, interior H\"older estimates have only been obtained in two dimensions \cite{LCMP} which we recall here: 

\begin{theorem}[Interior H\"older estimates, \cite{LCMP}]
\label{Holder_int_thm}
Assume $n=2$. Let $\phi \in C^2(\Omega)$ be a convex function satisfying 
$0 <\lambda \leq \det D^2\phi \leq \Lambda$ in $\Omega$.
Let $F:\Omega\rightarrow \R^n$ is a bounded vector field. Given a
section $S_\phi(x_0, 4h_0) \subset \subset \Omega$.
Let $p \in (1, \infty)$.
There exist a universal constant $\gamma >0$ depending only on $\lambda$ and $\Lambda$ 
and a constant $C >0$, depending only on $p$, $\lambda$, $\Lambda, h_0$ and $\diam(\Omega)$ with the following property.
For every solution $u$ to 
\begin{equation}
\label{main_eq}
\Phi^{ij} u_{ij}=\emph{div} F
\end{equation}
 in $S_{\phi}(x_0, 4h_0)$, and for all $x\in S_{\phi}(x_0, h_0)$,  we have the H\"older estimate:
 \begin{equation*}
  |u(x)-u(x_0)|\leq  C(p,\lambda,\Lambda, \emph{diam}(\Omega), h_0)  \left( 
 \|F \|_{L^{\infty}(S_\phi(x_0, 2h_0))} + \|u\|_{L^{p}(S_\phi(x_0, 2h_0))} \right)|x-x_0|^{\gamma}.
 \end{equation*}
\end{theorem}
In Theorem \ref{Holder_int_thm}, 
the section of  a convex function $\phi\in C^1(\overline{\Omega})$ at $x\in\overline\Omega$ with height $h$ is defined by
\begin{equation*}\label{def:section}
S_\phi(x,h) = \Big\{y\in \overline{\Omega}: \, \phi(y)< \phi(x) + D
\phi(x)\cdot (y-x) + h\Big\}.
\end{equation*}

When $F\equiv 0$, Theorem \ref{Holder_int_thm} was
established, in all dimensions, by Caffarelli and Guti\'errez in \cite{CG97}. 
Because $\Phi$ is 
divergence-free, that is, $ \sum_{i=1}^n \p_i \Phi^{ij}=0$ for all $j$, we can also write $\mathcal{L}_{\phi}$ as a divergence form 
operator: $$\mathcal{L}_{\phi} u=- \sum_{i, j=1}^{n} \partial_i (\Phi^{ij} u_{j}).$$ 
Thus, Theorem \ref{Holder_int_thm} can be viewed as an affine invariant version of 
related results by Murthy-Stampacchia \cite{MuSt} and Trudinger \cite{Tr} for second order elliptic equations in divergence form.
These authors studied the maximum principle, local and global estimates, local and global regularity for degenerate elliptic equations in the divergence form
\begin{equation}
\label{2divforme}
\small
\div (M(x)\nabla v(x))= \div F(x)~\text{in}~\Omega\subset\R^n.
\end{equation}
where $M(x)= (M_{ij}(x))_{1\leq i, j\leq n}$ is  nonnegative symmetric matrix, and
$F$ is a bounded vector field in $\R^n$. To obtain the H\"older regularity for solutions to (\ref{2divforme}), Murthy-Stampacchia and Trudinger required 
the high integrability 
of the eigenvalues of $M(x)$ and their inverses.
By Wang's counterexamples \cite{W95} to $W^{2,p}$ estimates for the Monge-Amp\`ere equations, this condition fails for 
the matrix $M=(\Phi^{ij})$ in Theorem 
\ref{Holder_int_thm} (even in two dimensions) when the ratio $\Lambda/\lambda$ is large.

A natural question regarding Theorem \ref{Holder_int_thm} is whether one can obtain the global H\"older estimates for solutions to (\ref{main_eq}) under suitable boundary conditions.
In this paper, we answer this question in the affirmative in two dimensions. Precisely, we obtain:
\begin{theorem}[Global H\"older estimates]
 \label{global-reg} Let $\Omega\subset \R^n$ be a bounded convex domain.
 Assume that $n=2$.
Assume that there exists a small constant $\rho>0$ such that
\begin{equation}
\label{cond1}
\Omega \subset B_{1/\rho}(0) ~\text{and for each } y\in\p\Omega~ \text{there is a ball } B_{\rho}(z)\subset \Omega~ \text{that is tangent to } \p\Omega ~\text{at } y. 
\end{equation}
Let  $\phi \in C^{0,1}(\overline 
\Omega) 
\cap 
C^2(\Omega)$  be a convex function satisfying
\begin{equation} 
\label{cond2}
0<\lambda\leq \det D^2 \phi \leq \Lambda \, \mbox{ in }\, \Omega.
\end{equation}
Assume further that 
on $\p \Omega$, $\phi$ 
separates quadratically from its 
tangent planes,  namely
\begin{equation}
\label{eq:sepa}
 \rho\abs{x-x_{0}}^2 \leq \phi(x)- \phi(x_{0})-D \phi(x_{0}) \cdot (x- x_{0})
 \leq \rho^{-1}\abs{x-x_{0}}^2, ~\text{for all}~ x, x_{0}\in\p\Omega.
\end{equation}
Let $u: \overline{\Omega}\rightarrow \R$ be a continuous function that 
solves the linearized Monge-Amp\`ere equation 
\begin{equation}
\label{bdr_div}
 \left\{
 \begin{alignedat}{2}
   \Phi^{ij}u_{ij} ~& = \div F ~&&\text{in} ~ \Omega, \\\
u &= \varphi ~&&\text{on}~\p \Omega,
 \end{alignedat} 
  \right.
\end{equation}
where $\varphi$ is a $C^{\alpha}$ function defined on $\p\Omega$ $(0<\alpha\leq 1)$ and $F\in L^{\infty}(\Omega)$. 
Then, there are positive constants $\alpha_1\in (0,1)$ and $K$ depending only on 
$\rho, \alpha, \lambda, \Lambda$ such that 
the following global H\"older estimates hold:
\begin{equation*}
 \|u\|_{C^{ \alpha_1}(\Omega)} \leq K\big( \|\varphi\|_{C^{\alpha}(\p \Omega)}+ \|F\|_{L^{\infty}(\Omega)}\big).
\end{equation*}
\end{theorem}
The global 
 H\"older estimates in Theorem \ref{global-reg} are an affine invariant and degenerate version of 
global H\"older estimates by Murthy-Stampacchia \cite{MuSt} and Trudinger \cite{Tr} for second order elliptic equations in divergence form. Our proof of Theorem \ref{global-reg} relies heavily on 
the new global high integrability 
of the gradient of Green's function of the 
 linearized Monge-Amp\`ere operator $\mathcal{L}_\phi$ in two dimensions. It is an open question whether our interior and global H\"older estimates for (\ref{bdr_div}) can be obtained in higher dimensions. 

We note from \cite[Proposition 3.2]{S2} that the quadratic separation \eqref{eq:sepa} holds 
for solutions to 
the Monge-Amp\`ere equations with the right hand side bounded away from $0$ and $\infty$ on uniformly convex domains and $C^3$ boundary data.

In the next theorem, we give an application of Theorem \ref{global-reg} to solvability of the second boundary value problem of singular, fourth order, fully nonlinear equations of Abreu type; see \cite{Le18} for a different approach using global
H\"older estimates for linearized Monge-Amp\`ere equation with right hand side having low integrability. 

\begin{theorem}
\label{SBV3}
Let $\Omega\subset\R^2$ be an open, smooth, bounded and uniformly convex domain.
Let $\varphi\in C^{\infty}(\overline{\Omega})$ and $\psi\in C^{\infty}(\overline{\Omega})$ with $\inf_{\p \Omega}\psi>0$. Then there exists a unique smooth,  uniformly convex solution $u\in C^{\infty}(\overline{\Omega})$ to
the following
second boundary value problem: 
\begin{equation}
  \left\{ 
  \begin{alignedat}{2}\sum_{i, j=1}^{2}U^{ij}w_{ij}~& =-|Du|^2\Delta u-2\sum_{i, j=1}^2 u_i u_i u_{ij}~&&\text{in} ~\Omega, \\\
 w~&= (\det D^2 u)^{-1}~&&\text{in}~ \Omega,\\\
u ~&=\varphi~&&\text{on}~\p \Omega,\\\
w ~&= \psi~&&\text{on}~\p \Omega.
\end{alignedat}
\right.
\label{Abreu3}
\end{equation}
Here $(U^{ij})$ is the cofactor matrix of $D^2 u$, that is, $(U^{ij})= (\det D^2 u) (D^2 u)^{-1}$.
\end{theorem}

The rest of the paper is devoted to proving Theorems \ref{global-reg} and \ref{SBV3}.  We use a priori estimates and degree theory to prove Theorem \ref{SBV3}. Note that the right hand side of the first equation in (\ref{Abreu3}) is the p-Laplacian with $p=4$:
$$-|Du|^2\Delta u-2u_i u_i u_{ij} = -\div (|Du|^2 Du).$$
We can assume that all functions $\phi$, $u$ in this paper are smooth. However, our estimates do not depend on the assumed smoothness but only on the given structural constants.

The analysis in this paper will be involved with $\Omega$ and $\phi$ satisfying either the global conditions (\ref{cond1})-(\ref{eq:sepa}) or  the following local conditions (\ref{om_ass})-(\ref{eq_u1}).

Let $\Omega\subset \R^{n}$ be a bounded convex set with
\begin{equation}\label{om_ass}
B_\rho(\rho e_n) \subset \, \Omega \, \subset \{x_n \geq 0\} \cap B_{\frac 1\rho} (0),
\end{equation}
for some small $\rho>0$ where we denote $e_n:= (0,\dots, 0, 1)\in \R^n$. Assume that 
\begin{equation}
\text{for each~} y\in\p\Omega\cap B_\rho(0), ~\text{there is a ball~} B_{\rho}(z)\subset \Omega \text{ that is tangent to } \p 
\Omega \text{ at } y.
\label{tang-int}
\end{equation}
Let 
 $\phi \in C^{0,1}(\overline 
\Omega) 
\cap 
C^2(\Omega)$  be a convex function satisfying
\begin{equation}\label{MAbound}
 0 <\lambda \leq \det D^2\phi \leq \Lambda \quad \text{in $\Omega$}.
\end{equation}
We assume that on $\p \Omega\cap B_\rho(0)$, 
$\phi$ separates quadratically from its tangent planes on $\p \Omega$, that is,
if $x_0 \in 
\p \Omega \cap B_\rho(0)$ then
\begin{equation}
 \rho\abs{x-x_{0}}^2 \leq \phi(x)- \phi(x_{0})-D\phi(x_{0}) \cdot (x- x_{0}) \leq 
\rho^{-1}\abs{x-x_{0}}^2\quad \text{ for all } x \in \p\Omega\cap \{x_n\leq \rho\}.
\label{eq_u1}
\end{equation}

We will use the letters $c, c_1, C, C_1, C', ...$, etc, to denote  {\it universal } constants that depend only on 
the structural constants $n, \rho, \lambda, \Lambda$  and/or $\alpha$. They may 
change from line to line. 

In Section \ref{G_sec}, we establish global $W^{1, 1+\kappa}$ estimates for the Green's function of the linearized Monge-Amp\`ere operator $\mathcal{L}_{\phi}$ under (\ref{pinch1}).
In Section \ref{L_sec}, we establish $L^{\infty}$ bounds and H\"older estimates at the boundary for solutions to (\ref{bdr_div}). The proofs of Theorems \ref{global-reg} and \ref{SBV3} will be given in Section \ref{P_sec}.

\section{Global $W^{1, 1+\kappa}$ estimates for the Green's function}

\label{G_sec}

Let $G_V(x,y)$ be the Green's function of $\mathcal{L}_{\phi}$ in $V$ with pole $y\in V\cap\Omega$ where $V\subset\overline{\Omega}$; that is $G_V(\cdot, y)$ is a positive solution of 
\begin{equation*}
\left\{\begin{array}{rl}
\mathcal{L}_{\phi} G_V( \cdot, y) &=\delta_y \qquad \mbox{ in}\quad V\cap\Omega,\\
G_V( \cdot, y) &=0\ \ \ \ \ \ \ \mbox{on}\quad \partial V
\end{array}\right.
\end{equation*}
with $\delta_y$ denoting the Dirac measure giving unit mass to the point $y$. 

Let $c_{\ast}= c_{\ast}(n, \lambda,\Lambda,\rho)$ be a small universal constant appearing in the global high integrability estimate \cite[inequality (5.12)]{LN3}  for the Green's function of $\mathcal{L}_{\phi}$.

Our main tools in this paper are following global estimates for the gradient of the Green's function of the linearized Monge-Amp\`ere operator $\Phi^{ij}\p_{ij}$
when the Monge-Amp\`ere measure $\det D^2\phi$ is only bounded away from zero and infinity. These estimates are the global version of those in \cite{L} in two dimensions; see also \cite{M} for related interior results in higher dimensions. 
\begin{theorem}[Global $W^{1,1+\kappa}$ estimates for the Green's function]
\label{DG_lem} 
There exists a universal constant $\kappa= \kappa(n, \lambda,\Lambda)>\frac{2-n}{3n-2}$ with the following property.\\
(i) Assume that $\Omega$ and $\phi$  satisfy (\ref{cond1})-(\ref{eq:sepa}). Then,  
\begin{equation*}
\int_{\Omega} |\nabla_x G_\Omega(x,y)|^{1+\kappa}\, dx\leq  C(n, \lambda,\Lambda,\rho) ~\text{for all } y\in\Omega
\end{equation*}
(ii) Assume $\Omega$ and $\phi$  satisfy  \eqref{om_ass}--\eqref{eq_u1}. If  $A= \Omega\cap B_\delta(0)$ where 
$\delta\leq c_{\ast}$, then
\begin{equation*}
\int_{A} |\nabla_x G_A(x,y)|^{1+\kappa}\, dx\leq  C(n, \lambda,\Lambda,\rho)~\text{for all } y\in A.
\end{equation*}
(iii) Let $V$ be either $\Omega$ as in (i) or $A$ as in (ii).  Let $\bar \kappa \in (0, \frac{1}{n-1})$. Then there is a positive number $\gamma (n,\bar \kappa)>0$ such that if 
$$|\Lambda/\lambda-1|<\gamma$$
then
\begin{equation}
\label{GVn_ineq}
\int_{V} |\nabla_x G_V(x,y)|^{1+\bar \kappa}\, dx\leq  C(n, \lambda,\Lambda, \bar\kappa, \rho) ~\text{for all } y\in V.
\end{equation}
\end{theorem}
\begin{rem}
\begin{myindentpar}{1cm} 
(i) In two dimensions, Theorem \ref{DG_lem} establishes the global $W^{1,1+\kappa}$ estimates ($\kappa>0$) for the Green's function of the linearized Monge-Amp\`ere operator $\Phi^{ij}\p_{ij}$
when the Monge-Amp\`ere measure $\det D^2\phi$ is only bounded away from zero and infinity. \\
(ii) On the other hand, in any dimension $n\geq 2$, Theorem \ref{DG_lem} establishes the global $W^{1,1 + \kappa}$ estimates 
for any $1+\kappa$ close to $\frac{n}{n-1}$
 for the Green's function of the linearized Monge-Amp\`ere operator $\Phi^{ij}\p_{ij}$
when the Monge-Amp\`ere measure $\det D^2\phi$ is close to a constant. In other words, if the Monge-Amp\`ere measure $\det D^2\phi$ is continuous then the Green's function of the linearized Monge-Amp\`ere operator $\Phi^{ij}\p_{ij}$
has the same global integrability, up to the first order derivatives, as the Green's function of the Laplace operator.
\end{myindentpar}
\end{rem}
\begin{proof} [Proof of Theorem \ref{DG_lem}] We first prove (i) and (ii).
Let $V$ be either $\Omega$ as in (i) or $A$ as in (ii) of the theorem.  
We need to show that
\begin{equation}
\label{GV_ineq}
\int_{V} |\nabla_x G_V(x,y)|^{1+\kappa}\, dx\leq  C(n, \lambda,\Lambda,\rho) ~\text{for all } y\in V~\text{and for some } \kappa (n, \lambda,\Lambda)>\frac{2-n}{3n-2}.
\end{equation}

{\it Step 1:} We first assert that for some $\e_\ast (n,\frac{\Lambda}{\lambda})>0$,
\begin{equation}
\label{W2V}
\int_{V} |D^2\phi|^{1+\e_\ast} dx \leq C(n, \lambda,\Lambda,\rho).
\end{equation}
Indeed, by De Philippis-Figalli-Savin's and Schmidt's $W^{2, 1+\e}$ estimates for the Monge-Amp\`ere equation \cite{DPFS, Sch} (see also \cite[Theorem 4.36]{Fi}), there exists $ \e_\ast(n,\frac{\Lambda}{\lambda})>0$ 
such that
$D^2 \phi\in L^{1+\e_\ast}_{loc}(\Omega)$. 
Using Savin's technique in his proof of the global $W^{2,p}$ estimates \cite{S3} for the Monge-Amp\`ere equation, we can show that (see also \cite[Theorem 5.3]{Fi}): If $\Omega$ and $\phi$  satisfy (\ref{cond1})-(\ref{eq:sepa}), then
$$\int_{\Omega} |D^2\phi|^{1+\e_\ast} dx \leq C(n, \lambda,\Lambda,\rho);$$
and if $\Omega$ and $\phi$  satisfy  \eqref{om_ass}--\eqref{eq_u1}, then
$$\int_{A} |D^2\phi|^{1+\e_\ast} dx \leq C(n, \lambda,\Lambda,\rho).$$
In all cases, we have (\ref{W2V}) as asserted.

Let $0<\e<\e_\ast$ be any positive number depending on $n$ and $\frac{\Lambda}{\lambda}$ and let
\begin{equation}
\label{pdef}
p= \frac{2+ 2\e}{2+\e}\in (1,2).
\end{equation}
Fix $y\in V$. 
Let $v(x):= G_V(x, y)$.
Let $k>0$. We use $v_j$ to denote the partial derivative $\p v/\p x_j$.

{\it Step 2:}  We have the following integral estimate
\begin{equation}
\label{keq}
\int_{\{x\in V: v(x)\leq k\}} \Phi^{ij} v_i v_j dx= k.
\end{equation}
To prove (\ref{keq}),
we use the truncation function $ w= -(v-k)^{-} + k$ to avoid the singularity of $v$ at $y$. By the assumption on the smoothness of $\phi$, $v$ is smooth away $y$. Thus $v\in W^{1, q}_{loc}(V\setminus \{y\})$ for all $q$.
Note that $w= v$ on $\{x\in V: v(x)\leq k\}$ while $w=k$ on $\{x\in V:v(x)\geq k\}$. 
Thus, $w\in W^{1, q}_0 (V)$ for all $q$.
Moreover, from the definition of  $v(x)= G_V(x, y)$, we find that
$$\int_V \Phi^{ij} v_i w_j dx=w(y)= k.$$
Using that $w_j= v_j$ on $\{x\in V: v(x)\leq k\}$ while $w_j=0$ on $\{x\in V: v(x)\geq k\}$, we obtain (\ref{keq}).

{\it Step 3:} We next claim that
\begin{equation}
\label{pk_est}
\int_{\{x\in V:v(x)\leq k\}} |Dv|^p dx\leq C(n, \lambda,\Lambda,\rho) k^{p/2}.
\end{equation}

To prove (\ref{pk_est}), we use (\ref{keq})
together with the arguments in \cite{L, LCMP}.
For completeness, we include its short proof here. Let $$S=\{x\in V: v\leq k\}. $$
We will use the following inequality
$\Phi^{ij} v_{i}(x) v_{j}(x)  \geq \frac{\det D^2 \phi |\nabla v|^2}{\Delta \phi}$
whose simple proof can be found in \cite[Lemma 2.1]{CG97}. 
It follows from (\ref{keq}) and $\det D^2\phi\geq\lambda$ that
$$\int_S \frac{|\nabla v|^2}{\Delta\phi}dx\leq \lambda^{-1} k.$$
Now, since $1<p<2$, using the H\"older inequality to $|\nabla v|^p= \frac{|\nabla v|^p}{(\Delta \phi)^{\frac{p}{2}}}  \left((\Delta\phi)^{\frac{p}{2}} \right)$
with exponents $\frac{2}{p}$ and $\frac{2}{2-p}$, we have
\begin{eqnarray}
\|\nabla v\|_{L^p(S)} \leq \left[\int_S \frac{|\nabla v|^2}{\Delta\phi}dx\right]^{\frac{1}{2}} \left(\int_S(\Delta \phi)^{\frac{p}{2-p}} 
dx\right)^{\frac{2-p}{2p}} \leq \lambda^{-1/2}k^{1/2} \|(\Delta\phi)\|_{L^{\frac{p}{2-p}}(S)}^{\frac{1}{2}}.
\label{Lp_est}
\end{eqnarray}
Applying (\ref{Lp_est}) to $p$ defined in (\ref{pdef}), noting that 
$ \frac{p}{2-p}= 1+\e,$ and recalling (\ref{W2V}), we obtain
\begin{eqnarray*}\|\nabla v\|_{L^p(S)} \leq \lambda^{-1/2}k^{1/2}\|(\Delta\phi)\|^{\frac{1}{2}}_{L^{1+\e}(S)}
\leq  \lambda^{-1/2} k^{1/2}\|\Delta\phi\|^{\frac{1}{2}}_{L^{1+\e_\ast}(S)} |S|^{\frac{\e_\ast-\e}{2(1+\e_\ast)(1+\e)}} 
\leq k^{1/2}C(n, \lambda,\Lambda,\rho) .
\end{eqnarray*}
The proof of (\ref{pk_est}) is complete.

{\it Step 4:} For any $1<q<\frac{n}{n-2},$ we have
 \begin{equation}\|v\|_{L^q(S)}\leq  \|v\|_{L^q(V)}\leq C(n, \lambda,\Lambda, \rho, q) .
  \label{Glq}
 \end{equation}
Indeed, let $q':=\frac{q}{q-1}$. 
If $\Omega$ and $\phi$  satisfy  \eqref{om_ass}--\eqref{eq_u1} and if $A= \Omega\cap B_\delta(0)$ where $\delta\leq c_{\ast}$ then 
estimate (5.12) in \cite{LN3} gives
\begin{eqnarray*}
\int_{A}G_{A}^q(x, y) dx =\int_{A}G_{A}^q(y, x) dx \leq C(n, \lambda,\Lambda,\rho, q) |A|^{\frac{3}{4} (\frac{2}{n}-\frac{1}{q'})} \leq   C(n, \lambda,\Lambda,\rho, q).
\end{eqnarray*}
On the other hand, if $\Omega$ and $\phi$  satisfy (\ref{cond1})-(\ref{eq:sepa}), then
by Corollary 2.6 in \cite{Le15}, we have
$$\int_{\Omega}G_{\Omega}^q(x, y) dx =\int_{\Omega}G_{\Omega}^q(y, x) dx \leq C(n, \lambda,\Lambda,\rho, q).$$
From the preceding estimates, we obtain (\ref{Glq}).

 As a consequence of (\ref{Glq}) and Chebyshev's inequality, we have
 \begin{equation}
 \label{Cheb}
 |\{x\in V: v(x)\geq k\}|\leq \frac{C_q(n, \lambda,\Lambda, \rho)}{k^{q}}.
 \end{equation}

{\it Step 5:} Now, we pass from the truncation of level $k$ to global estimates.
For any $\eta>0$, we have
\begin{eqnarray*}
\{x\in V: |\nabla v(x)|\geq \eta\}\subset \{x\in V: v\geq k\} \cup \{x\in V: |\nabla v(x)|\geq \eta; v(x)\leq k \}.
\end{eqnarray*}
By using (\ref{Cheb}) and (\ref{pk_est}), we obtain
\begin{eqnarray*}
|\{x\in V: |\nabla v(x)|\geq \eta\}| \leq \frac{C_q}{ k^{q}} + \int_{\{x\in V:v(x)\leq k\}} \frac{|\nabla v|^p}{\eta^{p}} dx\leq  \frac{C_q}{ k^{q}}  +  \frac{C_p k^{p/2}}{\eta^{p}}. 
\end{eqnarray*}
We choose $k$ such that $\eta^{p}= k^{\frac{p}{2} + q}$ or $k= \eta^{\frac{2p}{p+2q}}$. Then
\begin{equation*}
|\{x\in V: |\nabla v(x)|\geq \eta\}|  \leq \frac{C'_q}{k^q} =\frac{C_q^{'}}{\eta^{\frac{2pq}{p+2q}}}.
\end{equation*}
It follows from the layer cake representation that $|Dv|\in L^{1+\kappa}(V)$ for any $\kappa\in\R$ with $1+\kappa<\frac{2pq}{p+2q}.$
The proof of (\ref{GV_ineq}) will be complete if we can choose a suitable $1<q<\frac{n}{n-2}$ to make $\frac{2pq}{p+2q}> \frac{2n}{3n-2}$ so as to choose $\kappa>\frac{2-n}{3n-2}$ in the above inequality. This is possible,
since $$\lim_{q\rightarrow \frac{n}{n-2}}\frac{2pq}{p+ 2q}= \frac{2pn}{(n-2) p + 2n}> \frac{2n}{3n-2}$$ where the last inequality follows from $p>1$. In conclusion, we can find
$\kappa(n, \lambda,\Lambda)>\frac{2-n}{3n-2}$ such that 
\begin{equation*}
\int_{V} |\nabla_x G_V(x,y)|^{1+\kappa}\, dx\leq  C(n, \lambda,\Lambda,\rho).
\end{equation*}
The proof of (\ref{GV_ineq}) is complete.

Finally, we prove (iii). Let $\bar \kappa \in (0, \frac{1}{n-1})$.  From {\it Step 5} above, we find that,  in order to have (\ref{GVn_ineq}), it suffices to choose $\lambda$ and $\Lambda$ such that 
\begin{equation}
\label{pc1}
\frac{2pn}{(n-2) p + 2n}>1+\bar \kappa
\end{equation}
for some $$p= \frac{2+2\e}{2+\e} (\text{so that }\e = \frac{2p-2}{2-p})$$ where $0<\e<\e_\ast(n,\frac{\Lambda}{\lambda})$ with $\e_\ast$ as in {\it Step 1}. 

A direct calculation shows that (\ref{pc1}) holds as long as 
$2>p>p_0$
where 
$$p_0:= \frac{2n(1+\bar\kappa)}{2n-(1+\bar \kappa)(n-2)}<2.$$
The last inequality is due to the fact that $\bar \kappa \in (0, \frac{1}{n-1})$. Thus, we need to choose $\lambda$ and $\Lambda$ such that 
$$\e_\ast>\frac{2p_0-2}{2-p_0}.$$
This is always possible if 
$|\Lambda/\lambda-1|<\gamma$
for some small positive number $\gamma= \gamma(p_0, n)=\gamma (n,\bar\kappa)$; see \cite[Theorem 5.3]{Fi}.
\end{proof}

\section{$L^{\infty}$ bounds and H\"older estimates at the boundary}

\label{L_sec}
 
In this section we establish $L^{\infty}$ bounds in Lemma \ref{global-ball} and H\"older estimates at the boundary in Proposition \ref{local-H} for solutions to (\ref{bdr_div}).
Let $c_\ast=c_{\ast}(\lambda,\Lambda,\rho)>0$ be as in Section \ref{G_sec}.

As a consequence of Theorem \ref{DG_lem}, we first have the following global estimates for solutions to inhomogeneous
linearized Monge--Amp\`ere equations (\ref{bdr_div}) in two dimensions.

 \begin{lemma}\label{global-ball} Assume that $n=2$.  Consider the following settings:
 \begin{myindentpar}{1cm}
 (i) Assume that $\Omega$ and $\phi$  satisfy (\ref{cond1})-(\ref{eq:sepa}).\\
 (ii) Assume $\Omega$ and $\phi$  satisfy  \eqref{om_ass}--\eqref{eq_u1}.  Let
$A= \Omega\cap B_\delta(0)$ where 
$\delta\leq c_\ast$.
\end{myindentpar}
Let $V$ be either $\Omega$ as in (i) or A as in (ii). Assume that $F\in L^{\infty}(V)$  and  $u\in W^{2,n}_{loc}(V)\cap C(\overline{V})$ satisfies
\[
\mathcal{L}_{\phi} u\leq \div F\quad \mbox{almost everywhere in}\quad V.
\]
Then  there exist positive constants $\kappa_2 (\lambda,\Lambda)$ and $C(\lambda,\Lambda,\rho)$ such that
\[
\sup_{V}{u} \leq \sup_{\partial V}{u^+} + C |V|^{\kappa_2}   \|F\|_{L^{\infty}(V)}. 
\]
\end{lemma}
\begin{proof}
Let $\kappa=\kappa(\lambda,\Lambda)>0$ be as in Theorem \ref{DG_lem}. Set
$$\kappa_2:= 1-\frac{1}{1+\kappa}>0.$$
Using H\"older inequality to the estimates in Theorem \ref{DG_lem}, we find $C(\lambda,\Lambda,\rho)>0$ such that
\begin{equation}
\label{DG_L1}
\int_{V} |\nabla_x G_V(x,y)|\, dx\leq  C(\lambda,\Lambda,\rho) |V|^{\kappa_2} .
\end{equation}
Let $G_V(x,y)$ be the Green's function of $\mathcal{L}_{\phi}$ in $V$ with pole $y\in V$.
Define
\[
v(x) := \int_{V} G_V(x,y) \div F(y)\, dy \quad\mbox{for}\quad x\in V.
\]
Then $v$ is a solution of 
$$\mathcal{L}_{\phi} v =\div F \mbox{ in }V,~\text{and}~
v =0\mbox{ on } \partial V.$$
Since $\mathcal{L}_{\phi}(u-v)\leq 0$ in $V$, we obtain from the Aleksandrov-Bakelman-Pucci (ABP) maximum principle (see \cite[Theorem~9.1]{GiT}) that
\begin{equation}\label{comp_prin}
u(x) \leq \sup_{\partial V}{u^+} + v(x)\quad \mbox{in} \quad V.
\end{equation}

As the operator $\mathcal{L}_\phi$ can be written in the divergence form with symmetric coefficient, we  infer 
 from \cite[Theorem~1.3]{GW}
that $G_V(x,y)=G_V(y,x)$ for all $x,y\in V$. Thus, using (\ref{DG_L1}), we can estimate for all $x\in V$
\begin{eqnarray}
v(x)&=&\int_{V} G_V(x,y) \div F(y)\, dy=\int_{V} G_V(y,x) \div F(y)\, dy\nonumber\\
&=& -\int_{V} \nabla_y G_V(y,x) F(y)\, dy\leq  C(\lambda,\Lambda,\rho) |V|^{\kappa_2}   \|F\|_{L^{\infty}(V)}.
\label{v_est}
\end{eqnarray}
The desired estimate follows from (\ref{comp_prin}) and (\ref{v_est}).
\end{proof}

Next, we obtain the following H\"older estimates at the boundary for solutions to inhomogeneous
linearized Monge--Amp\`ere equations (\ref{bdr_div}) in two dimensions.

\begin{proposition}  \label{local-H}
Assume $\Omega$ and $\phi$  satisfy  \eqref{om_ass}--\eqref{eq_u1}. 
Assume that $n=2$. Let $\kappa_2$ be as in Lemma \ref{global-ball}.
Let $u \in C\big(B_{\rho}(0)\cap 
\overline{\Omega}\big) \cap W^{2,n}_{loc}(B_{\rho}(0)\cap 
\Omega)$  be a  solution to 
\begin{equation*}
 \left\{
 \begin{alignedat}{2}
   \Phi^{ij}u_{ij} ~& = \div F ~&&\text{ in } ~ B_{\rho}(0)\cap \Omega, \\\
u &= \varphi~&&\text{ on }~\p \Omega \cap B_{\rho}(0),
 \end{alignedat} 
  \right.
\end{equation*} 
where $\varphi\in C^{\alpha}(\partial\Omega\cap B_{\rho}(0))$ for some $\alpha\in (0,1)$ and $F\in L^{\infty}(\Omega\cap B_\rho(0))$.  Let
 $$\alpha_0 :=\min\big\{\alpha, \kappa_2\big\}.$$
Then, there exist positive constants $\delta $ and $C$ depending only $\lambda, \Lambda, \alpha, \rho$ 
such that, for any $x_{0}\in\partial\Omega\cap B_{\rho/2}(0)$  and for all $x\in \Omega\cap B_{\delta}(x_{0})$, we have
$$|u(x)-u(x_{0})|\leq C|x-x_{0}|^{\frac{\alpha_0}{\alpha_0 +3n}} \Big(
\|u\|_{L^{\infty}(\Omega\cap B_{\rho}(0))} + \|\varphi\|_{C^\alpha(\partial\Omega\cap B_{\rho}(0))}  
+ \|F\|_{L^{\infty}(\Omega\cap B_{\rho}(0))} \Big). $$
\end{proposition}

\begin{proof} Our proof relies on Lemma \ref{global-ball} and a construction of suitable barriers as in the proof of  Proposition 5.1 in \cite{LN3}. We omit the details.
\end{proof}

\section{ Global H\"older  Estimates and Singular Abreu equations} 
\label{P_sec}
In this section, we prove Theorems \ref{global-reg} and \ref{SBV3}.  Theorem \ref{global-reg} follows from Theorem \ref{Holder_int_thm}, Lemma \ref{global-ball} and Theorem \ref{global-H} below.

\begin{theorem}\label{global-H}
Assume $\Omega$ and $\phi$  satisfy  \eqref{om_ass}--\eqref{eq_u1}. 
Assume that $n=2$. 
Let $u \in C\big(B_{\rho}(0)\cap 
\overline{\Omega}\big) \cap W^{2,n}_{loc}(B_{\rho}(0)\cap 
\Omega)$  be a  solution to 
\begin{equation*}
 \left\{
 \begin{alignedat}{2}
   \Phi^{ij}u_{ij} ~& = \div F ~&&\text{ in } ~ B_{\rho}(0)\cap \Omega, \\\
u &= \varphi~&&\text{ on }~\p \Omega \cap B_{\rho}(0),
 \end{alignedat} 
  \right.
\end{equation*} 
where $\varphi\in C^{\alpha}(\partial\Omega\cap B_{\rho}(0))$ for some $\alpha\in (0,1)$ and $F\in L^{\infty}(\Omega\cap B_\rho(0))$. Then, there exist positive constants $\beta$ and $C$ depending only on $\lambda, \Lambda, \alpha, \rho$
such that 
$$|u(x)-u(y)|\leq C|x-y|^{\beta}\Big(
\|u\|_{L^{\infty}(\Omega\cap B_{\rho}(0))} + \|\varphi\|_{C^\alpha(\partial\Omega\cap B_{\rho}(0))}  + \|F\|_{L^{\infty}(\Omega\cap B_{\rho}(0))} \Big)~\text{for all }x, y\in \Omega\cap 
B_{\frac{\rho}{2}}(0). $$
\end{theorem}
\smallskip

\begin{proof}[Proof of Theorem \ref{global-H}]
The proof of the global H\"older 
estimates in this theorem is similar to the proof  of \cite[Theorem~1.7]{LN3}.
It combines the boundary H\"older estimates in Proposition~\ref{local-H} and the interior H\"older continuity estimates  in Theorem \ref{Holder_int_thm} using
Savin's Localization Theorem
\cite{ S2}. Thus we omit the details.
\end{proof}

We are now in a position to complete the proof of Theorem \ref{global-reg}.
\begin{proof}[Proof of Theorem \ref{global-reg}]
From Lemma \ref{global-ball}, we find that 
\begin{equation*}
\|u\|_{L^{\infty}(\Omega)} \leq \|\varphi\|_{L^{\infty}(\p\Omega)} + C(\lambda,\Lambda,\rho)   \|F\|_{L^{\infty}(\Omega)}.
\end{equation*}
The desired global H\"older estimates in Theorem \ref{global-reg} follow from  combining Theorems \ref{Holder_int_thm} and \ref{global-H}.
\end{proof}

Finally, we give a proof of Theorem \ref{SBV3}.

\begin{proof}[Proof of Theorem \ref{SBV3}]
The proof of the uniqueness of solutions is similar to that of Lemma 4.5 in \cite{Le18} so we omit it.
 The existence proof uses a priori estimates and degree theory as in Theorem 2.1 in \cite{Le18}. Here, 
 we only focus on proving the a priori estimates  for $u$ in $C^{k}(\overline{\Omega})$ for any $k\geq 2$. 
 {\it Step 1: positive bound from below and above for $\det D^2 u$.}

First, by the convexity of $u$, we have $$U^{ij} w_{ij}=-|Du|^2\Delta u-2u_i u_i u_{ij}\leq 0~\text{in}~\Omega.$$ By the maximum principle, the function $w$ attains its minimum value on $\p\Omega$. It follows that $$w\geq \inf_{\p\Omega}\psi:= c>0~\text{in }\Omega.$$ 
 Therefore,
 \begin{equation}\label{detu_up}\det D^2 u=w^{-1}\leq C_1:= c^{-1}~\text{in }\Omega.
 \end{equation}
 Now, we can construct an explicit barrier using the uniform convexity of $\Omega$  and 
  the upper bound for $\det D^2 u$
  to  show that 
  \begin{equation}
  \label{Dubound}
  |Du|\leq C_2 ~\text{in }\Omega
  \end{equation} for a constant $C_2$ depending only on $\Omega$, $\varphi$ and $\inf_{\p\Omega}\psi$. 
 
 Noting that we are in two dimensions so $ \text{trace } (U^{ij}) =\Delta u$. We compute in $\Omega$
$$U^{ij} (w + 2C_2^2|x|^2)_{ij}= -|Du|^2\Delta u-2u_i u_i u_{ij} + 4C_2^2 \Delta u\geq 0.
$$
By the maximum principle, $w(x) + 2C_2^2 |x|^2$ attains it maximum value on the boundary $\p\Omega$. Recall that $w=\psi$ on $\p\Omega$.
Thus, for all $x\in\Omega$, we have
\begin{equation}
\label{wmaxi}
w(x) \leq w(x) + 2C_2^2 |x|^2 \leq \max_{\p\Omega} (\psi + 2C_2^2|x^2|)\leq C_3.
\end{equation}
It follows from (\ref{detu_up}) and (\ref{wmaxi}) that
\begin{equation}\label{detD2u} C^{-1}\leq w=(\det D^2 u)^{-1}\leq C
\end{equation}
where $C$ depends only on $\Omega$, $\varphi$ and $\inf_{\p\Omega}\psi$. 

{\it Step 2: higher order derivative estimates for $u$.}
From (\ref{detD2u}) and (\ref{Dubound}), we apply the global H\"older estimates for the linearized Monge-Amp\`ere equation in Theorem \ref{global-reg}
 to $$U^{ij} w_{ij}=-|Du|^2\Delta u-2u_i u_i u_{ij}  = - \div (|Du|^2 Du)~ \text{in }\Omega$$ with boundary value $w=\psi\in C^{\infty}(\p\Omega)$ on $\p\Omega$ to 
  conclude that $w\in C^{\alpha}(\overline{\Omega})$
  with 
  \begin{equation}
  \label{walpha}
  \|w\|_{C^{\alpha}(\overline{\Omega})}\leq  C\left(\|\psi\|_{C^{1}(\p\Omega)} + \| D u\|^3_{L^{\infty}(\Omega)}\right)\leq C_4
  \end{equation}
  for universal constants $\alpha\in (0, 1)$ and $C_4>0$. Now, we note that $u$ solves the Monge-Amp\`ere equation
  $$\det D^2 u= w^{-1}$$
  with right hand side being in $C^{\alpha}(\overline{\Omega})$ and boundary value $\varphi\in C^3(\p\Omega)$ on $\p\Omega$.
  Therefore, by the global $C^{2,\alpha}$ estimates for the Monge-Amp\`ere equation \cite{TW08}, we have $u\in C^{2,\alpha}(\overline{\Omega})$
  with universal estimates
  \begin{equation}
  \label{u2alpha}
   \|u\|_{C^{2,\alpha}(\overline{\Omega})}\leq C_5~\text{and } C_5^{-1} I_2\leq D^2 u\leq C_5 I_2.
   \end{equation}
   As a consequence, the second order operator $U^{ij} \p_{ij}$ is uniformly elliptic with H\"older continuous coefficients.  A bootstrap argument 
   for the equation 
   $$U^{ij} w_{ij}=-|Du|^2\Delta u-2u_i u_i u_{ij}$$
   concludes the proof of the a priori estimates  for $u$ in $C^{k}(\overline{\Omega})$ for any $k\geq 2$. 
\end{proof}

\bibliographystyle{plain}

\end{document}